\documentclass[10pt,oneside,letterpaper,final]{amsart}
\usepackage[T1]{fontenc}
\usepackage[latin9]{inputenc}
\usepackage{calrsfs}
\usepackage{geometry}
\usepackage{amssymb}
\usepackage{esint}

\theoremstyle{definition}
	\newtheorem{defn}{Definition}[section]

	\theoremstyle{plain}
	\newtheorem{thm}[defn]{Theorem}
	\newtheorem*{thm*}{Theorem}
	\newtheorem{lem}[defn]{Lemma}
	\newtheorem*{lem*}{Lemma}
	\newtheorem{prop}[defn]{Proposition}
	\newtheorem*{prop*}{Proposition}
	\newtheorem{cor}[defn]{Corollary}
	\newtheorem*{cor*}{Corollary}
	\theoremstyle{remark}
	\newtheorem{rmk}[defn]{}

	\DeclareMathOperator{\Ker}{Ker}
	\newcommand{\ms}{\text{ }}

	\newcommand{\C}{\mathbb C}
	\newcommand{\Z}{\mathbb Z}
	\newcommand{\N}{\mathbb N}

	\newcommand{\mb}[1]{\mathbb{#1}}

	\newcommand{\frk}[1]{\mathfrak{#1}}
	
	\newcommand{\norm}[1]{\left\| #1 \right\|}
	\newcommand{\vnorm}{\norm{\hspace{.15cm}}}
	\newcommand{\projnorm}[1]{\norm{#1}^{\wedge}}
	\newcommand{\vprojnorm}{\projnorm{\ms}}
	
	\newcommand{\ttimes}[1]{\ms \widetilde{*}_{#1} \ms}
	
\renewcommand{\labelenumi}{(\roman{enumi})}
\title[Analytic Subordination]{Analytic subordination results in free probability from non-coassociative derivation-comultiplications\\}
\author{Stephen Curran}
\address{Department of Mathematics\\University of California at Berkeley\\Berkeley, CA 94720}
\email{curransr@math.berkeley.edu}

\begin{document}
\begin{abstract}
We extend Voiculescu's approach to analytic subordination through the coalgebra of the free difference quotient to non-coassociative derivation-comultiplications appearing in free probability theory.  We obtain new proofs of Voiculescu's analytic subordination results for freely Markovian triples, and for multiplication of unitaries which are free with amalgamation.
\end{abstract}

\maketitle
\section*{Introduction}

\noindent A derivation-comultiplication on a unital algebra $A$ over $\C$ is a linear map
\begin{equation*}
 \Delta:A \to A \otimes A,
\end{equation*}
which satisfies the product rule $\Delta(ab) = (a\otimes 1)\Delta(b) + \Delta(a)(1\otimes b)$.  Derivation-comultiplications play a prominent role in free probability theory, most notably in Voiculescu's ``microstates-free'' approaches to free entropy, free Fisher information and free mutual information (\cite{fisher},\cite{mutual}).  Of particular interest is the free difference quotient, introduced to study free Fisher information and free entropy, and at the center of the ``free analysis'' of Voiculescu (\cite{cyclomorphy}, \cite{free1}, \cite{free2}).

The free difference quotient $\partial_{X:B}$ is the derivation-comultiplication on $B \langle X \rangle$ determined by
\begin{align*}
 \partial_{X:B}(X) &= 1 \otimes 1,\\
\partial_{X:B}(b) &= 0, \ms \ms \ms \ms \ms\ms \ms (b \in B),
\end{align*}
where $B$ is a unital algebra over $\C$ and $X$ is algebraically free from $B$.  $\partial_{X:B}$ has the additional property of coassociativity, i.e. 
\begin{equation*}
\left(\text{id} \otimes \partial_{X:B}\right) \circ \partial_{X:B} = \left(\partial_{X:B} \otimes \text{id}\right) \circ \partial_{X:B}.
\end{equation*}
In considering the corepresentations of this coalgebra, Voiculescu found a natural explanation for the phenomenon of analytic subordination, a powerful tool in free harmonic analysis.

In \cite{entropyi}, Voiculescu proved (under some easily removed genericity assumptions) that if $X$ and $Y$ are self-adjoint and freely independent random variables, then the Cauchy transforms of $G_{X+Y}$ and $G_X$ satisfy an analytic subordination relation in the upper half-plane.  He used this result to prove certain inequalities on $p$-norms of densities, free entropies and Riesz energies.  It was later discovered by Biane that the subordination extends to the operator-valued resolvents, and that a similar result holds for free multiplicative convolution \cite{biane}.  He used these results to prove certain Markov-transitions properties for processes with free increments.  (See also \cite{newapproach},\cite{romuald},\cite{nica} for other approaches to subordination in free probability).  

Though technically useful, the proofs of these results did little to explain why analytic subordination appears in the context of free convolutions. What Voiculescu observed in \cite{coalgebra} is that, roughly speaking, the invertible corepresentations of $\partial_{X:B}$ are the $B$-resolvents $(b-X)^{-1}$ (and their matricial generalizations).  Moreover, if $X$ and $Y$ are $B$-freely independent, then a certain conditional expectation is a coalgebra morphism from the coalgebra of $\partial_{X+Y:B}$ to the coalgebra of $\partial_{X:B}$.  Since coalgebra morphisms preserve corepresentations, one should expect that $B$-resolvents of $X+Y$ are mapped to $B$-resolvents of $X$ by this conditional expectation.  This approach led to the generalization of the earlier results for free additive convolution to the $B$-valued context.  

In \cite{markovian}, Voiculescu found that he could extend this result by simple operator-valued analytic continuation arguments.  Here he found a general subordination result for freely Markovian triples, and gave a $B$-valued extension of Biane's result for multiplicative convolution of unitaries.
 
In this paper we extend Voiculescu's method to non-coassociative derivation-comultiplications appearing in free probability theory.  Through this approach we obtain new proofs of the subordination results in \cite{markovian}.

In his development of free mutual information (\cite{mutual}), Voiculescu introduced the derivation
\begin{equation*}
 \delta_{A:B}:A \vee B \to A \vee B \otimes A \vee B
\end{equation*}
determined by
\begin{align*}
 \delta_{A:B}(a) &= a\otimes 1 - 1\otimes a, \ms \ms \ms (a \in A),\\
\delta_{A:B}(b) &= 0, \ms \ms \ms \ms \ms \ms \ms \ms \ms\ms\ms\ms\ms\ms\ms\ms\ms\ms (b \in B),
\end{align*}
where $A$ and $B$ are unital algebras which are algebraically free, and $A \vee B$ denotes the algebra generated by $A \cup B$.  Here we will use the coalgebra structure of $\delta_{A:B}$ to give a natural proof of the following subordination result for freely Markovian triples.

\begin{thm*}
Let $(M,\tau)$ be a W$^*$-algebra with faithful, normal trace state $\tau$.  Let $1 \in B \subset M$ be a W$^*$-subalgebra, and let $1 \in A,C \subset M$ be $*$-subalgebras which are $B$-free in $(M,E_B)$, i.e. $A,B,C$ is \textit{freely Markovian}.  Then there is an analytic function $F:\mb H_+(A) \times \mb H_+(C) \to B$ such that
\begin{equation*}
 E_{A \vee B} (a+c)^{-1} = (a+F(a,c))^{-1}
\end{equation*}
for $a \in \mb H_+(A)$, $c \in \mb H_+(C)$.
\end{thm*}

To develop ``microstates-free'' free Fisher information and free entropy for unitaries (\cite{mutual}), Voiculescu introduced the derivation
\begin{equation*}
 d_{U:B}: B\langle U,U^* \rangle \to B \langle U,U^* \rangle \otimes B \langle U,U^* \rangle,
\end{equation*}
determined by
\begin{align*}
 d_{U:B}(U) &= 1 \otimes U, \\
d_{U:B}(U^*) &= -U^* \otimes 1, \\
d_{U:B}(b) &= 0, \ms\ms\ms\ms\ms\ms\ms\ms\ms\ms\ms \ms \ms b \in B,
\end{align*}
where $U$ is a unitary that is algebraically free from the unital algebra $B$.  Here we show that the coalgebra of $d_{U:B}$ is the natural object in the following subordination result for multiplication of $B$-freely independent unitaries.

\begin{thm*}
Let $(M,\tau)$ be a W$^*$-algebra with faithful, normal trace state $\tau$.  Let $1 \in B \subset M$ be a W$^*$-subalgebra, and let $U,V \in M$ such that $B \langle U,U^* \rangle$ is $B$-free from $B\langle V,V^* \rangle$ in $(M,E_B)$.  Then there is an analytic function $F:\mb D(B) \to \mb D(B)$ such that
\begin{equation*}
 E_{B \langle U,U^* \rangle} UVb(1-UVb)^{-1} = UF(b)(1-UF(b))^{-1}
\end{equation*}
and $\norm{F(b)} \leq \norm{b}$ for $b \in \mb D(B)$.
\end{thm*}
\noindent The resolvents $Ub(1-Ub)^{-1}$ appearing here are related to the $S$-transform in free probability theory, see (\cite{frv},\cite{dykema}).

The idea behind these proofs is quite simple.  Because $\delta$ and $d$ are not coassociative, we cannot expect to find interesting corepresentations for these comultiplications.  However, the resolvents $(a+b)^{-1}$ and $Ub(1-Ub)^{-1}$ appearing above are characterized by certain relations with $\delta$ and $d$, respectively.  Moreover, these relations are preserved by certain conditional expectations which are coalgebra morphisms for $\delta$ and $d$.  We should expect then that these resolvents are preserved by these conditional expectations.  The technical difficulties that arise are in working with the closures of these unbounded derivations.

Besides this introduction, the paper has five sections.

Section 1 is purely algebraic.  We look at the relationship between derivations and certain resolvents in a general setting.

In Section 2 we show that certain conditional expectations are coalgebra morphisms for $\delta$ and $d$.

In Section 3 we extend some technical results from \cite{mutual} to the operator-valued case, which will be needed in the next section.

Section 4 contains the proof of the analytic subordination result for freely Markovian triples.  The greatest difficulty is in proving that certain elements in the kernel of the closure of $\delta_{A:B}$ are actually in $B$.  

Section 5 covers the analytic subordination result for multiplication of $B$-freely independent unitaries.  The approach is similar to the freely Markovian case, but the technical difficulties are slightly easier.

\medskip
\noindent \textit{Acknowledgments:}  I would like to thank Dan-Virgil Voiculescu for suggesting this problem, and for his guidance and support while completing this paper.  I would also like to thank Dimitri Shlyakhtenko for his helpful suggestions.

\section{Derivations and Resolvents}

\noindent Here we discuss the relationship between derivations and certain resolvents in a general algebraic framework.

\begin{rmk}\label{derivations}
Let $A,B$ be unital algebras over $\C$, and $\varphi_1,\varphi_2:A \to B$ be unital homomorphisms.  A linear map $D:A \to B$ is a derivation with respect to the $A$-bimodule structure defined by $\varphi_1,\varphi_2$ if 
\begin{equation*}
 D(a_1a_2) = \varphi_1(a_1)D(a_2) + D(a_1)\varphi_2(a_2).
\end{equation*}
It is easy to see that this implies $D(1) = 0$, and if $a \in A$ is invertible then
\begin{equation*}
 D\left(a^{-1}\right) = -\varphi_1\left(a^{-1}\right)D(a)\varphi_2\left(a^{-1}\right).
\end{equation*}

\end{rmk}

\begin{prop}\label{derresolvents}
Let $A,B,\varphi_1,\varphi_2,D$ be as above and let $N = \Ker{D}$.
\begin{enumerate}
 \item Fix $a \in A$.  If $\alpha \in A$ is invertible, and $D(\alpha) = - \varphi_1(\alpha)D(\alpha)\varphi_2(\alpha)$, then  $\alpha = (a+n)^{-1}$ for some $n \in \Ker D$.  Conversely, if $n \in \Ker D$ is such that $a+n$ is invertible, then
\begin{equation*}
 D\left((a+n)^{-1}\right) = -\varphi_1\left((a+n)^{-1}\right)D(a)\varphi_2\left((a+n)^{-1}\right).
\end{equation*}
\item Suppose $U \in A$ is invertible, and $D(U) = \varphi_2(U)$.  If $\alpha \in A$ is such that $1 + \alpha$ is invertible and $D(\alpha) = \varphi_1(\alpha+1)\varphi_2(\alpha)$, then $\alpha = Un(1-Un)^{-1}$ for some $n \in N$ such that $1-Un$ is invertible.  Conversely, if $n \in N$ is such that $1-Un$ is invertible, then
\begin{equation*}
 D\left(Un(1-Un)^{-1}\right) = \varphi_1\left(Un(1-Un)^{-1}+1\right)\varphi_2\left(Un(1-Un)^{-1}\right).
\end{equation*}
\end{enumerate}

\end{prop}

\begin{proof}\hfill
\begin{enumerate}
 \item Fix $a \in A$ and suppose $\alpha \in A$ satisfies the hypotheses, then
\begin{equation*}
 D\left(\alpha^{-1}\right) = -\varphi_1(\alpha^{-1})D(\alpha)\varphi_2(\alpha)^{-1} = D(a).
\end{equation*}
So $\alpha^{-1} - a \in \Ker D$ which proves one direction, the converse is trivial.
\item Suppose $U \in A$ is invertible, and $\alpha \in A$ satisfies the hypotheses, then
\begin{align*}
 D\left(U^{-1}(\alpha+1)^{-1}\right) &= -\varphi_1\left(U^{-1}\right)\varphi_1\left((\alpha+1)^{-1}\right)D(\alpha+1)\varphi_2\left((\alpha+1)^{-1}\right)\\
&-\varphi_1\left(U^{-1}\right)D(U)\varphi_2\left(U^{-1}\right)\varphi_2\left((\alpha+1)^{-1}\right)\\
&= -\varphi_1\left(U^{-1}\right)\left[\varphi_2(\alpha) + 1\right]\varphi_2\left((\alpha+1)^{-1}\right)\\
&= D\left(U^{-1}\right).
\end{align*}
So $n = U^{-1} - U^{-1}(\alpha+1)^{-1} \in \Ker D$, and hence $\alpha = (1-Un)^{-1} - 1 = Un(1-Un)^{-1}$.  The converse is a simple computation.  
\end{enumerate}

\end{proof}

\begin{rmk}
In the sequel, we will apply Proposition \ref{derresolvents} to certain completions of $\delta$ and $d$.
\begin{enumerate}
 \item Observe that $B \subset \Ker \delta_{A:B}$, so if $a \in A$, $b \in B$ are such that $a + b$ is invertible in the completion of $A \vee B$, then $\alpha = (a+b)^{-1}$ satisfies the hypotheses of (i) above.  
\item Likewise, $B \subset \Ker d_{U:B}$, so if $b \in B$ is such that $(1-Ub)$ is invertible in the completion of $B\langle U,U^* \rangle$, then $\alpha = Ub(1-Ub)^{-1}$ satisfies the conditions of (ii) above.
\end{enumerate}

\end{rmk}

\section{Coalgebra morphisms in free probability}

\noindent In this section we prove that certain conditional expectations arising in the contexts of free Markovianity, and $B$-free multiplicative convolution of unitaries, are coalgebra morphisms for the comultiplications $\delta$ and $d$, respectively.  Because we will need these results in the next section, we will work with operator-valued generalizations of $\delta$ and $d$.
\begin{rmk}
In the remainder of the paper, $(M,\tau)$ will denote a tracial W$^*$-probability space.  If $A,B \subset M$, $A \vee B$ will denote the algebra generated (algebraically) by $A \cup B$.  If $1 \in A \subset M$ is a $*$-subalgebra, $E_A$ will denote the canonical trace preserving conditional expectation of $M$ onto W$^*(A)$.  
\end{rmk}

\begin{rmk}
Suppose that $1 \in B \subset M$ is a W$^*$-subalgebra, and that $1 \in A_1,A_2 \subset M$ are subalgebras containing $B$ which are algebraically free with amalgamation over $B$.  Letting $A = A_1 \vee A_2$ denote the algebra generated by $A_1$ and $A_2$, define
\begin{equation*}
 \delta_{A_1:A_2;B} : A \to A \otimes_B A
\end{equation*}
to be the derivation into the $A$-bimodule $A \otimes_B A$, which is determined by
\begin{equation*}
 \delta_{A_1:A_2;B} = \begin{cases} a \otimes 1 - 1 \otimes a, & \text{if }a \in A_1,\\
                       0, & \text{if }a \in A_2.
                      \end{cases}
\end{equation*}
The \textit{$B$-valued liberation gradient} $j = j(A_1:A_2;B)$ is then defined by the requirements that $j \in L^2(A)$, and
\begin{equation*}
 E_{B}(ja) = (E_B \otimes E_B)(\delta_{A_1:A_2;B}(a)), \ms \ms \ms \ms \ms (a \in A).
\end{equation*}
Except in Section 3, we will be interested only in the case $B = \C$, in which case we recover the definitions of Voiculescu in \cite{mutual} of $\delta(A_1:A_2)$ and of the \textit{liberation gradient} $j(A_1:A_2)$.  This $B$-valued generalization was introduce by Nica, Shlyakhtenko and Speicher in \cite{NShSp} as a method for studying $B$-freeness of the algebras $A_1$ and $A_2$.
\end{rmk}

\begin{rmk}
Suppose $1 \in B \subset M$ is a W$^*$-subalgebra, $A \subset M$ is a subalgebra containing $B$ and $U \in M$ is a unitary such that $B\langle U,U^* \rangle$ is algebraically free with amalgamation over $B$ from $A$.  Define
\begin{equation*}
 d_{U:A;B}: A\langle U,U^* \rangle \to A \langle U,U^* \rangle \otimes_B A \langle U,U^* \rangle
\end{equation*}
to be the derivation determined by
\begin{align*}
 d_{U:A;B}(U) &= 1 \otimes U,\\
d_{U:A;B}(U^*) &= -U^* \otimes 1,\\
d_{U:A;B}(a) &= 0, \ms \ms\ms\ms\ms\ms\ms\ms\ms\ms\ms \ms (a \in A).
\end{align*}
The \textit{conjugate of $U$ relative to $A$ with respect to $B$}, denoted $\xi = \xi(U:A;B)$, is then defined by the requirements that $\xi \in L^2(A\langle U,U^* \rangle)$ and
\begin{equation*}
 E_{B}\left(\xi m\right) = \left(E_B \otimes E_B\right) \left(d_{U:A;B}(m)\right), \ms \ms \ms \ms m \in A\langle U,U^* \rangle.
\end{equation*}
We will mostly be interested in the case $B = \C$, in which case we recover the definition of $d_{U:B}$ from \cite{mutual}.  This $B$-valued generalization was considered by Shlyakhtenko in \cite{Shl}.  
\end{rmk}

\begin{rmk}
The following lemma is an operator-valued generalization of a result in \cite{coalgebra}.  The proof is an easy adaptation of the argument found there, we include it here for the convenience of the reader.
\end{rmk}

\begin{lem} \label{opexpmorph}
Let $1 \in B_1,B \subset M$ be W$^*$-subalgebras in $(M,\tau)$ such that $B_1 \subset B$.  Let $1 \in A,C \subset M$ be $*$-subalgebras which are $B$-free in $(M,E_B)$.  Let $D:A \vee B \vee C \to (A \vee B \vee C) \otimes_{B_1} (A \vee B \vee C)$ be a derivation such that $D(B \vee C) = 0$ and $D(A \vee B) \subset (A \vee B) \otimes_{B_1} (A \vee B)$.  Then
\begin{equation*}
 \left(E_{A \vee B} \otimes_{B_1} E_{A \vee B}\right) \circ D = D \circ E_{A\vee B}|_{A \vee B \vee C}.
\end{equation*}

\end{lem}

\begin{proof}
First note that $B$-freeness implies
\begin{equation*}
 E_{A \vee B}(A \vee B \vee C) \subset A \vee B.
\end{equation*}

Let $F_1 = (A \vee B) \cap \Ker E_B$, $F_2 = (B \vee C) \cap \Ker E_B$.  Since $A \vee B$ and $B \vee C$ are $B$-free, we have
\begin{equation*}
 (A \vee B \vee C) \ominus (A \vee B) = F_2 \oplus \bigoplus_{k \geq 2} \bigoplus_{\substack{\alpha_1 \neq \dotsb \neq \alpha_k\\ \alpha_i \in \{1,2\}}} F_{\alpha_1}F_{\alpha_2}\dotsb F_{\alpha_k},
\end{equation*}
where the orthogonal difference and direct sums are with respect to the $B_1$-valued inner product defined by $E_{B_1}$.  Now $DF_2 = 0$, and $DF_1 \subset (F_1 + B) \otimes (F_1 + B)$ by hypothesis.  If $\alpha_1 \neq \dotsb \neq \alpha_k$, $\alpha_i \in \{1,2\}$, $k \geq 2$, then
\begin{equation*}
 D\left(F_{\alpha_1}\dotsb F_{\alpha_k}\right) \subset \sum_{\substack{1 \leq i \leq k\\ \alpha_i = 1}} F_{\alpha_1} \dotsb F_{\alpha_{i-1}}\left(F_1 + B\right) \otimes_{B_1} \left(F_1 + B\right)F_{\alpha_{i+1}}\dotsb F_{\alpha_k}.
\end{equation*}
If $k \geq 2$, either $i > 1$ or $i < k$ so that either
\begin{equation*}
 E_{A \vee B}\left(F_{\alpha_1}\dotsb F_{\alpha_{i-1}}\left(F_1 + B\right)\right) = 0
\end{equation*}
or
\begin{equation*}
 E_{A \vee B}\left(\left(F_1+B\right)F_{\alpha_{i+1}}\dotsb F_{\alpha_k}\right) = 0.
\end{equation*}

Since also $DF_2 = 0$, we have shown that
\begin{equation*}
 \left(E_{A \vee B} \otimes_{B_1} E_{A \vee B} \right) \left(D(A \vee B \vee C \ominus A \vee B)\right) = 0.
\end{equation*}
Since
\begin{equation*}
 \left(E_{A \vee B} \otimes_{B_1} E_{A \vee B} \right) \circ D \circ E_{A \vee B}|_{A \vee B \vee C} = D \circ E_{A \vee B}|_{A \vee B \vee C}
\end{equation*}
by hypothesis, the result follows.
\end{proof}

\begin{cor}\label{deltamorph}
Let $1 \in B_1, B \subset M$ be W$^*$-subalgebras such that $B_1 \subset B$, and let $1 \in A,C \subset M$ be $*$-subalgebras which are $B$-free in $(M,E_B)$ and such that $A$ is algebraically free from $B \vee C$ with amalgamation over $B_1$.  Then
\begin{equation*}
 \left(E_{A \vee B} \otimes_{B_1} E_{A \vee B}\right) \otimes \delta_{A:B \vee C;B_1} = \delta_{A:B;B_1} \circ E_{A\vee B}|_{A \vee B \vee C}.
\end{equation*}

\end{cor}

\begin{cor}\label{libmorph}
Suppose that $j(A:B;B_1)$ exists, then so does $j(A:B\vee C;B_1)$ and
\begin{equation*}
 j(A:B \vee C;B_1) = j(A:B;B_1).
\end{equation*}

\end{cor}

\begin{proof}
For $m \in A \vee B \vee C$, we have
\begin{align*}
 E_{B_1}\left(j(A:B;B_1)m\right) &= E_{B_1}\left(j(A:B;B_1)E_{A \vee B}(m)\right)\\
&= \left(E_{B_1} \otimes E_{B_1}\right) \delta_{A:B;B_1}\left(E_{A\vee B}(m)\right)\\
&= \left(E_{B_1} \otimes E_{B_1}\right) \delta_{A:B\vee C;B_1}(m).
\end{align*}

\end{proof}

\begin{cor}\label{dmorph}
Let $1 \in B_1, B \subset M$ be W$^*$-subalgebras such that $B_1 \subset B$.  Let $U,V \in M$ be unitaries which are $B$-freely independent, and such that $U$ is algebraically free from $B\langle V,V^* \rangle$ with amalgamation over $B_1$.  Then $E_{B \langle U,V,U^*,V^* \rangle} \subset B \langle U,U^* \rangle$, and 
\begin{equation*}
\left(E_{B\langle U,U^* \rangle} \otimes_{B_1} E_{B\langle U,U^*\rangle} \right) \circ d_{UV:B;B_1} = d_{U:B;B_1} \circ E_{B \langle U,U^* \rangle}|_{B \langle UV,V^*U^*\rangle}.
\end{equation*}

\end{cor}

\begin{proof}
Apply Lemma \ref{opexpmorph} to find that $E_{B\langle U,U^* \rangle} B\langle U,V,U^*,V^* \rangle \subset B \langle U,U^* \rangle$, and
\begin{equation*}
 \left(E_{B \langle U,U^* \rangle} \otimes_{B_1} E_{B \langle U,U^* \rangle}\right) \circ d_{U:B\langle V,V^*\rangle;B_1} = d_{U:B;B_1} \circ E_{B \langle U,U^* \rangle}|_{B \langle U,V,U^*,V^*\rangle }.
\end{equation*}
Since $d_{U:B\langle V,V^*\rangle;B_1}|_{B \langle UV,V^*U^* \rangle} = d_{UV:B;B_1}$, the result follows by restricting to $B \langle UV,V^*U^*\rangle$.  
\end{proof}

\begin{cor}\label{conjmorph}
Suppose that $\xi(U:B;B_1)$ exists, then so does $\xi(UV:B;B_1)$, and
\begin{equation*}
 \xi(UV:B;B_1) = E_{B\langle UV,V^*U^*\rangle} \left(\xi(U:B;B_1)\right).
\end{equation*}

\end{cor}

\begin{proof}
The proof is similar to Corollary \ref{libmorph}.
\end{proof}

\section{Regularization via unitary conjugation}

\noindent Our aim in this section is to show that if $1 \in A,B \subset M$ are $*$-subalgebras, then we can can find a unitary $U$ arbitrarily close to the identity such that $W^*(UAU^* \vee B) \cap W^*(A \vee B) = B$, which will be needed in the next section.  In the case $B = \C$, this follows easily from the considerations in \cite{mutual}.  Here we extend the necessary results from that paper to the $B$-valued case by using the $B$-valued liberation gradient introduced in the previous section.

\begin{rmk}
The $L^2$-norm of the $B$-valued liberation gradient gives a measure of how far the algebras $A_1$ and $A_2$ are from being $B$-free.  In particular, it is shown in \cite{NShSp} that $A_1$ and $A_2$ are $B$-free if and only if $j(A_1:A_2;B) = 0$.  In the case $B = \C$, Voiculescu gave some estimates on the ``distance'' between the algebras $A_1$ and $A_2$ when the liberation gradient $j(A_1:A_2)$ is bounded \cite{mutual}.  We begin by observing that his estimates extend directly to the $B$-valued case.
\end{rmk}

\begin{lem}
Let $1 \in B \subset M$ be a W$^*$-subalgebra, and let $1 \in A_1,A_2 \subset M$ be $*$-subalgebras which contain $B$, and such that $A_1$ is algebraically free from $A_2$ with amalgamation over $B$.  Suppose that $j(A_1:A_2;B)$ exists.  If $m \in A_1 \cap \Ker E_B$, $m' \in A_2 \cap \Ker E_B$ then
\begin{equation*}
 E_B(j(A_1:A_2;B)mm') = - E_B(j(A_1:A_2;B)m'm) = - E_B(mm')
\end{equation*}
and
\begin{equation*}
 E_B(j(A_1:A_2;B)[m,m']) = -2E_B(mm').
\end{equation*}
In particular,
\begin{equation*}
 \tau(j(A_1:A_2;B)[m,m']) = -2\tau(mm').
\end{equation*}

\end{lem}
\qed
\begin{prop}\label{graddist}
Suppose that $\norm{j(A_1:A_2;B)} < \infty$.  If $m \in A_1 \cap \Ker E_B$, $m' \in A_2 \cap \Ker E_B$ then
\begin{equation*}
 |\tau(mm')| \leq \frac{\norm{j(A_1:A_2;B)}}{\left(1 + \norm{j(A_1:A_2;B)}^2\right)^{1/2}}|m|_2 |m'|_2.
\end{equation*}
Equivalently, 
\begin{equation*}
 \norm{(E_{A_1} - E_B)(E_{A_2} - E_B)} \leq \frac{\norm{j(A_1:A_2;B)}}{\left(1 + \norm{j(A_1:A_2;B)}^2\right)^{1/2}}.
\end{equation*}

\end{prop}

\begin{proof}
Identical to \cite[Proposition 7.2]{mutual}.
\end{proof}

\begin{rmk}
We now turn to the existence of the $B$-valued liberation gradient $j(A_1:A_2;B)$ after conjugating by a unitary in $M$ which commutes with $B$.  As observed in the scalar case by Voiculescu, the key is the relation between $\delta$ and $d$.
\end{rmk}

\begin{prop}
Let $1 \in B \subset M$ be a W$^*$-subalgebra, and $1 \in A \subset M$ a $*$-subalgebra which contains $B$.  If $U$ is a unitary in $M$ which commutes with $B$ and is algebraically free from $A$ with amalgamation over $B$, then
\begin{equation*}
 d_{U:A;B}|_{A \vee UAU^*} = -\delta_{UAU^*:A;B}.
\end{equation*}

\end{prop}

\begin{proof}
We have
\begin{align*}
d_{U:A;B}\left(a_1Ua_2U^*\dotsb a_{2k-1}Ua_{2k}U^*\right) &= \sum_{1 \leq p \leq k} \bigl(a_1Ua_2U^*\dotsb a_{2p-1}\otimes Ua_{2p}U^*\dotsb a_{2k-1}Ua_{2k}U^*\\
&-a_1Ua_2U^*\dotsb a_{2p-1}Ua_{2p}U^* \otimes a_{2p+1}\dotsb a_{2k-1}Ua_{2k}U^*\bigr)\\
&= - \delta_{UAU^*:A;B}\left(a_1Ua_2U^*\dotsb a_{2k-1}Ua_{2k}U^*\right).
\end{align*}

\end{proof}

\begin{cor}\label{libconj}
If $\xi(U:A;B)$ exists, then so does $j(UAU^*:A;B)$ and 
\begin{equation*}
j(UAU^*:A;B) = -E_{A \vee UAU^*}(\xi(U:A;B)).
\end{equation*}
\end{cor}
\qed

\begin{prop}\label{opconj}
Let $1 \in B \subset M$ be a $W^*$-subalgebra, and suppose that $U \in M$ is a unitary such that $\C[U,U^*]$ is independent from $B$.  Then if $\xi(U:\C;\C)$ exists, so does $\xi(U:B;B)$ and
\begin{equation*}
 \xi(U:B;B) = \xi(U:\C;\C).
\end{equation*}

\end{prop}

\begin{proof}
Since $U$ commutes with $B$, we just need to check that
\begin{equation*}
 E_B(\xi(U:\C;\C)U^n) = (E_B \otimes_B E_B)(d_{U:B;B}(U^n))
\end{equation*}
for all $n \in \Z$.  If $n \geq 0$, then by independence we have
\begin{align*}
 E_B(\xi(U:\C;\C)U^n) &= \tau(\xi(U:\C;\C)U^n)\\
&= (\tau \otimes_\C \tau)\left(d_{U:\C;\C}(U^n)\right)\\
&= \sum_{k = 0}^{n-1} \tau(U^k)\tau(U^{n-k})\\
&= \sum_{k=0}^{n-1}E_B(U^k)E_B(U^{n-k})\\
&= (E_B \otimes_B E_B) \left(d_{U:B;B}(U^n)\right).
\end{align*}
The case $n < 0$ is similar.

\end{proof}

\begin{prop}\label{freeconj}
Let $1 \in B \subset M$ be a $W^*$-algebra, $1 \in A \subset M$ a $*$-subalgebra containing $B$, and $U \in M$ a unitary such that $A$ is $B$-free from $B\langle U,U^* \rangle$ in $(M,E_B)$.  If $\xi(U:B;B)$ exists, then so does $\xi(U:A;B)$ and
\begin{equation*}
 \xi(U:A;B) = \xi(U:B;B).
\end{equation*}

\end{prop}

\begin{proof}
Apply Lemma \ref{opexpmorph} with $D,A,B_1,B,C$ replaced by $d_{U:A;B},B\langle U,U^* \rangle, B,B, A$ to find
\begin{equation*}
 \left(E_B \otimes_B E_B \right) \circ d_{U:B;B} \circ E_{B \langle U,U^* \rangle}|_{A \langle U,U^*\rangle} = \left(E_B \otimes_B E_B\right) \circ d_{U:A;B}.
\end{equation*}
Now for $m \in A\langle U,U^* \rangle$, we have
\begin{align*}
 E_B\left(\xi(U:B;B)m\right) &= E_B\left(\xi(U:B;B)E_{B \langle U,U^* \rangle}(m)\right)\\
&= \left(E_B \otimes_B E_B\right)d_{U:B;B}\left(E_{B \langle U,U^* \rangle}(m)\right)\\
&= \left(E_B \otimes_B E_B\right) d_{U:A;B}(m).
\end{align*}
\end{proof}

\begin{prop}\label{conjsemicirc}
Let $S$ be a $(0,1)$-semicircular random variable in $(M,\tau)$.  Fix $0 < \epsilon < 1$, and let $U_\epsilon = \text{exp}(\pi i \epsilon S)$.  Then $\xi(U_\epsilon:\C ;\C)$ exists, and
\begin{equation*}
 \norm{\xi(U_\epsilon:\C;\C) - i(2\pi^2\epsilon)^{-1}S} \leq \frac{\epsilon(2-\epsilon)}{2\pi(1-\epsilon)}.
\end{equation*}
In particular, $\xi(U_\epsilon:\C;\C) \in W^*(U_\epsilon)$. 
\end{prop}

\begin{proof}
The distribution of $U_\epsilon$ with respect to $\tau$ has density
\begin{equation*}
 p\left(e^{i\theta}\right) = \chi_{[-\pi\epsilon,\pi \epsilon]}  \frac{4}{\pi\epsilon^2}\sqrt{\epsilon^2 - \theta^2/\pi^2}
\end{equation*}
with respect to the normalized Lebesgue measure on $\mb T$.  By \cite[Proposition 8.7]{mutual}, $\xi(U_\epsilon:\C;\C)$ exists, and is given by $i(Hp)(U_\epsilon)$, where $Hp$ is the circular Hilbert transform of $p$, i.e. $Hp$ is the a.e. limit of $H_\delta p$ as $\delta \to 0$, where
\begin{equation*}
 (H_\delta p)\left(e^{i\theta_1}\right) = -\frac{1}{2\pi}\int_{\delta < |\theta| \leq \pi} p\left(e^{i(\theta_1-\theta)}\right)\text{cot}\left(\frac{\theta}{2}\right) \ms d\theta.
\end{equation*}
For $x \neq 0$, we have the expansion (\cite{tables})
\begin{equation*}
 \frac{1}{2}\text{cot}\left(\frac{x}{2}\right) = \frac{1}{x} + \sum_{n =1}^\infty \frac{1}{x + 2\pi n} - \frac{1}{2\pi n}.
\end{equation*}
It follows that for $0 < |\theta| \leq 2\pi \epsilon$, we have
\begin{align*}
\left|\frac{1}{2}\text{cot}\left(\frac{\theta}{2}\right) - \frac{1}{\theta}\right| &\leq \sum_{n \geq 1} \frac{|\theta|}{|\theta+2n\pi|2n\pi}\\
&\leq (2\pi \epsilon)\left(\frac{1}{2\pi(1-\epsilon)2\pi} + \frac{1}{2\pi}\sum_{n \geq 2} \frac{1}{2\pi(n-1)} - \frac{1}{2\pi n} \right)\\
&= \frac{\epsilon(2-\epsilon)}{2\pi(1-\epsilon)}.
\end{align*}
Hence if $|\theta_1| \leq \pi \epsilon$, then
\begin{equation*}
 \left| (H_\delta p)\left(e^{i\theta_1}\right) + \frac{1}{\pi} \int_{\delta < |\theta| \leq \pi} \frac{p\left(e^{i(\theta_1 - \theta)}\right)}{\theta} \ms d\theta \right| \leq \frac{\epsilon(2-\epsilon)}{2\pi(1-\epsilon)},
\end{equation*}
since $p(\text{exp}(i(\theta_1 - \theta))) = 0$ if $|\theta| > 2 \pi \epsilon$.  But 
\begin{equation*}
 -\frac{1}{\pi} \int_{\delta < |\theta| \leq \pi} \frac{p\left(e^{i(\theta_1 - \theta)}\right)}{\theta} \ms d\theta
\end{equation*}
converges as $\delta \to 0$ to the Hilbert transform of the semicircular law of radius $\pi \epsilon$ evaluated at $\theta_1$.  By the results in \cite[Section 3]{fisher}, this is equal to $\theta_1/(2\pi^3\epsilon^2)$.  So for $|\theta_1| \leq \pi \epsilon$, we have
\begin{equation*}
 \left|(Hp)\left(e^{i\theta_1}\right) - \frac{\theta_1}{2\pi^3\epsilon^2}\right| \leq \frac{\epsilon(2-\epsilon)}{2\pi(1-\epsilon)}.
\end{equation*}
It follows that 
\begin{equation*}
 \norm{\xi(U_\epsilon:\C;\C) - i(2\pi^2\epsilon)^{-1}S} \leq \frac{\epsilon(2-\epsilon)}{2\pi(1-\epsilon)}.
\end{equation*}

\end{proof}

\begin{cor} \label{split}
Let $1 \in B \subset M$ be a W$^*$-subalgebra, $1 \in A \subset M$ a $*$-subalgebra containing $B$, and $S$ a $(0,1)$-semicircular element in $(M,\tau)$ which is independent from $B$ and $B$-freely independent from $A$.  Then for $0 < \epsilon < 1$, we have
\begin{equation*}
 W^*(A \vee B) \cap W^*(U_\epsilon A U_\epsilon \vee B) = B,
\end{equation*}
where $U_\epsilon = \text{exp}(\pi i \epsilon S)$.  
\end{cor}

\begin{proof}
By Propositions \ref{opconj} and \ref{freeconj}, $\xi(U_\epsilon:A;B)$ exists and
\begin{equation*}
 \xi(U_\epsilon:A;B) = \xi(U_\epsilon:B;B) = \xi(U_\epsilon:\C;\C).
\end{equation*}
Applying Corollary \ref{libconj}, we see that $j(U_\epsilon AU_\epsilon^*:A;B)$ exists and
\begin{equation*}
 j(U_\epsilon A U_{\epsilon}^* : A ; B) = -E_{A \vee U_\epsilon AU_\epsilon^*} \left[\xi(U_\epsilon:\C;\C)\right].
\end{equation*}
By Proposition \ref{conjsemicirc}, $\xi(U_\epsilon:\C;\C)$ is bounded and hence
\begin{equation*}
 \norm{j(U_\epsilon AU_\epsilon^*:A;B)} < \infty.
\end{equation*}
The result now follows from Proposition \ref{graddist}.
\end{proof}

\section{Analytic subordination for freely Markovian triples}

\noindent In this section we use the derivation $\delta_{A:B}$ to prove the analytic subordination result for a freely Markovian triple $(A,B,C)$.  The main difficulty is in showing that certain ``smooth'' elements in the kernel of the closure of $\delta_{A:B}$ actually lie in $B$.

\begin{rmk}
Let $1 \in A, B \subset M$ be $*$-subalgebras which are algebraically free.  Let $A * B$ denote the $*$-algebra free product of $A$ and $B$ (with amalgamation over $\C$).  Given an invertible $S \in M$, there is a unique $*$-homomorphism $\rho_S:A * B \to M$ determined by
\begin{align*}
 \rho_S(a) &= SaS^{-1}, \ms \ms (a \in A),\\
 \rho_S(b) &= b, \ms \ms \ms \ms \ms \ms \ms \ms\ms\ms (b \in B).
\end{align*}
We will denote by $\rho$ the isomorphism of $A * B$ onto $A \vee B$.  
\end{rmk}

\begin{rmk}
Recall that $\delta_{A:B}: A \vee B \to A \vee B \otimes A \vee B$ is the derivation determined by
\begin{align*}
 \delta_{A:B}(a) &= a \otimes 1 - 1 \otimes a, \ms \ms (a \in A),\\
\delta_{A:B}(b) &= 0, \ms\ms\ms\ms\ms\ms\ms\ms\ms\ms\ms\ms\ms\ms\ms\ms\ms\ms (b \in B).
\end{align*}
For $p \geq 0$, we define $\delta_{A:B}^{(p)}:A \vee B \to (A \vee B)^{\otimes (p+1)}$ recursively by $\delta_{A:B}^{(0)} = \text{id}_{A \vee B}$ and
\begin{equation*}
 \delta_{A:B}^{(p+1)} = \left(\delta_{A:B} \otimes \text{id}^{\otimes p}\right)\circ \delta_{A:B}^{(p)}.
\end{equation*}

\end{rmk}

\begin{rmk} \label{deltanorm}
We will work a certain ``smooth'' Banach algebra completion of $A * B$.  Given $0 < R < 1$, define $\vnorm_R^{\sim}$ on $A * B$ by
\begin{equation*}
 \norm{f}_R^\sim = \sum_{p \geq 0} \projnorm{\delta_{A:B}^{(p)}(\rho(f))}_{(p+1)} R^p,
\end{equation*}
where $\vprojnorm_{(s)}$ denotes the norm on the projective tensor product $M^{\widehat \otimes s}$.  
\begin{lem*} $\vnorm{}_{R}^\sim$ is a finite norm on $A * B$, and if $f, g \in A * B$ then
\begin{equation*}
 \norm{fg}_R^\sim \leq \norm{f}_R^\sim\norm{g}^\sim_R.
\end{equation*}

\end{lem*}
\end{rmk}
\begin{proof}
Since $\delta_{A:B}$ is a derivation, if $f,g \in A * B$ then we have
\begin{equation*}
 \delta_{A:B}^{(p)}(\rho(fg)) = \sum_{k=0}^{p} \left(\delta_{A:B}^{(k)}(\rho(f)) \otimes 1^{\otimes (p-k)}\right)\left(1^{\otimes k} \otimes \delta_{A:B}^{(p-k)}(\rho(g))\right),
\end{equation*}
so that
\begin{align*}
 \norm{fg}_R^\sim &= \sum_{p \geq 0} \projnorm{\delta_{A:B}^{(p)}(\rho(fg))}_{(p+1)}R^p\\
&= \sum_{p \geq 0} \projnorm{\sum_{k =0}^p \left(\delta_{A:B}^{(k)}(\rho(f)) \otimes 1^{\otimes (p-k)}\right)\left(1^{\otimes k} \otimes \delta_{A:B}^{(p-k)}(\rho(g))\right)}_{(p+1)} R^p\\
&\leq \sum_{p \geq 0} \sum_{k =0}^p \projnorm{\delta_{A:B}^{(k)}(\rho(f))}_{(k+1)}R^k\projnorm{\delta_{A:B}^{(p-k)}(\rho(g))}_{(p-k+1)}R^{p-k}\\
&= \norm{f}^\sim_R\norm{g}_R^\sim.
\end{align*}
Since $\vnorm_R^\sim$ is easily seen to be finite when restricted to $A$ and to $B$, it follows that $\vnorm_R^\sim$ is a finite norm on $A * B$.
\end{proof}

\begin{rmk}
Let $A \ttimes R B$ denote the Banach algebra obtained by completing $A * B$ under $\vnorm_R^\sim$.  It is clear that $\rho$ extends to a contractive homomorphism $\widetilde \rho: A \ttimes R B \to C^*(A \vee B)$, note however that $\widetilde \rho$ need not be injective.   
\end{rmk}

\begin{rmk}
The main analytic tool we have for studying $\delta_{A:B}$ is its relation to $\rho_{(1-m)}$, $m \in M$, $\norm{m} < 1$.  To state this relation precisely, we will first need to introduce some notation.  Given $m_1,\dotsc,m_s \in M$, let $\theta_s[m_1,\dotsc,m_s]$ denote the linear map from $M^{\otimes (s+1)}$ into $M$ determined by
\begin{equation*}
 \theta_p[m_1,\dotsc,m_s](m'_1 \otimes \dotsc m'_{s+1}) = m'_1m_1m_2'\dotsb m_sm'_{s+1}.
\end{equation*}
Note that
\begin{equation*}
\norm{\theta_p[m_1,\dotsc,m_s](\xi)} \leq \norm{m_1}\dotsb\norm{m_s}\projnorm{\xi}_{(s+1)},
\end{equation*}
where $\vprojnorm_{(s+1)}$ denotes the projective tensor product norm on $M^{\widehat \otimes (s+1)}$.
\end{rmk}

\begin{prop}\label{deltahom}
If $f \in A * B$ and $m \in M$, $\norm{m} < 1$, then
\begin{equation*}
\rho_{(1-m)}(f) = \sum_{p \geq 0} \theta_p[m,\dotsc,m]\left(\delta_{A:B}^{(p)}(\rho(f))\right),
\end{equation*}
where the series converges absolutely in the uniform norm on $M$.  In particular, $\rho_{(1-m)}$ extends to a contractive homomorphism $\widetilde \rho_{(1-m)}:A \ttimes R B \to M$.
\end{prop}

\begin{proof}
First we will check that the series converges absolutely.  Indeed, by the remark above we have
\begin{equation*}
\sum_{p \geq 0} \norm{\theta_p[m,\dotsc,m] \left(\delta_{A:B}^{(p)} (\rho(f))\right)} \leq \sum_{p \geq 0} \norm{m}^p \projnorm{\delta_{A:B}^{(p)}(\rho(f))}_{(p+1)} = \norm{f}_{\norm{m}}^\sim,
\end{equation*}
which is finite by \ref{deltanorm}.  

Now let $\varphi(f)$ denote the right hand side, it suffices to show that $\varphi$ is a homomorphism from $A*B$ into $M$ which agrees with $\rho_{(1-m)}$ when restricted to $A$ or $B$.  If $f, g \in A * B$, then
\begin{align*}
 \varphi(fg) &= \sum_{p \geq 0} \theta_p[m,\dotsc,m]\left(\delta_{A:B}^{(p)}(\rho(fg))\right)\\
&= \sum_{p \geq 0} \theta_p[m,\dotsc,m]\sum_{k = 0}^p \left(\delta_{A:B}^{(k)}(\rho(f))\otimes 1^{\otimes (p-k)}\right)\left(1^{\otimes k} \otimes \delta_{A:B}^{(p-k)}(\rho(g))\right)\\
&= \sum_{p \geq 0} \sum_{k=0}^p \theta_k[m,\dotsc,m]\left(\delta_{A:B}^{(k)}(\rho(f))\right)\theta_{(p-k)}[m,\dotsc,m]\left(\delta_{A:B}^{(p-k)}(\rho(g))\right)\\
&= \varphi(f)\varphi(g).
\end{align*}
So $\varphi$ is indeed a homomorphism.  Clearly $\varphi(b) = b = \rho_{(1-m)}(b)$.  For $a \in A$, we have
\begin{align*}
 \varphi(a) &= \sum_{p \geq 0} \theta_p[m,\dotsc,m]\left(a \otimes 1^{\otimes p} - 1 \otimes a \otimes 1^{\otimes (p-1)}\right)\\
&= \sum_{p \geq 0}(am^p-mam^{p-1})\\
&= (1-m)a\sum_{p \geq 0}m^p\\
&= (1-m)a(1-m)^{-1}\\
&= \rho_{(1-m)}(a).
\end{align*}
Now if $\norm{m} \leq R < 1$, then we have
\begin{equation*}
 \norm{\rho_{(1-m)}(f)} \leq \norm{f}_{\norm{m}}^\sim \leq \norm{f}_R^\sim,
\end{equation*}
so that $\rho_{(1-m)}$ extends by continuity to a contractive homomorphism $\widetilde \rho_{(1-m)}:A \ttimes R B \to M$.
\end{proof}

\begin{rmk}\label{gradclos}
Recall that the liberation gradient $j(A:B)$ is determined by $j(A:B) \in L^1(W^*(A \vee B))$ and
\begin{equation*}
 \tau\left(j(A:B)m\right) = (\tau \otimes \tau)\left(\delta_{A:B}(m)\right) \ms \ms \ms \ms m \in A \vee B.
\end{equation*}
Voiculescu has shown \cite{mutual} that the existence of $j(A:B)$ in $L^2(W^*(A \vee B))$ is a sufficient condition for the closability of $\delta_{A:B}$, viewed as an unbounded operator
\begin{equation*}
 \delta_{A:B}: L^2(W^*(A \vee B)) \to L^2(W^*(A \vee B) \otimes W^*(A \vee B)).
\end{equation*}
In particular, $|j(A:B)|_2 < \infty$ implies that $\delta_{A:B}$ is closable in the uniform norm, we will denote this closure by $\overline \delta_{A:B}$.  We will need the following standard result on closable derivations (\cite{bratelli}, \cite{coalgebra}).
\end{rmk}

\begin{prop}\label{closder}
Let $K,L$ be unital C$^*$-algebras, let $\varphi_1,\varphi_2:K \to L$ be unital $*$-homomorphisms, let $1 \in A \subset K$ be a unital $*$-subalgebra, and let $D:A \to L$ be a closable derivation with respect to the $A$-bimodule structure on $L$ defined by $\varphi_1,\varphi_2$.  The closure $\overline D$ is then a derivation, and the domain of definition $\frk D(\overline D)$ is a subalgebra.  Moreover, if $a \in A$ is invertible in $K$, then $a^{-1} \in \frk D(\overline D)$ and
\begin{equation*}
 \overline D\left(a^{-1}\right) = -\varphi_1\left(a^{-1}\right)D(a)\varphi_2\left(a^{-1}\right).
\end{equation*}

\end{prop}

\begin{prop}\label{deltaker}
Let $1 \in B \subset M$ be a W$^*$-subalgebra, and $1 \in A \subset M$ a $*$-subalgebra such that $A$ and $B$ are algebraically free.  Suppose also that $|j(A:B)|_2 < \infty$.  If $0 < R < 1$, then $\widetilde \rho(A \ttimes R B) \subset \frk D(\overline \delta_{A:B})$.  Furthermore, if $f \in A \ttimes R B$ and $\overline \delta_{A:B}(\widetilde \rho(f)) = 0$, then $\widetilde \rho(f) \in B$.  
\end{prop}

\begin{proof}
It is clear from the definition of the norm $\vnorm_{R}^\sim$ that $\widetilde \rho$ maps $A \ttimes R B$ into $\frk D(\overline \delta_{A:B})$.  Suppose then that $f \in A \ttimes R B$, and $\overline \delta_{A:B}(\widetilde \rho(f)) = 0$.  Let $f_n \in A * B$ s.t. $f_n \to f$ in $A \ttimes R B$.  Then 
\begin{align*}
\lim_{n \to \infty} \delta_{A:B}\left( \rho(f_n)\right) = \overline \delta_{A:B} \left(\widetilde \rho(f)\right) = 0,
\end{align*}
the limit being taken in the projective tensor product norm $\vprojnorm_{(2)}$.  Since $(\delta_{A:B} \otimes \text{id})$ is closable, it follows that
\begin{equation*}
\lim_{n \to \infty} \projnorm{\delta_{A:B}^{(2)}(\rho(f_n))}_{(3)} = 0.
\end{equation*}
Iterating, we see that
\begin{equation*}
 \lim_{n \to \infty} \projnorm{\delta_{A:B}^{(p)}(\rho(f_n))}_{(p+1)} = 0
\end{equation*}
for all $p \geq 0$.  Let $m \in M$, $\norm{m} < R$.  Since $f_n \to f$ in $A \ttimes R B$, it follows that $\norm{f_n}_{R}^\sim \leq C$, where $C$ is a constant which does not depend on $n$.  Given $\epsilon > 0$, find $P$ such that 
\begin{equation*}
 C \frac{(\norm{m}/R)^P}{1-(\norm{m}/R)} < \epsilon.
\end{equation*}
Then find $N$ such that $n \geq N$ implies
\begin{equation*}
 \sum_{p = 1}^{P-1} \projnorm{\delta_{A:B}^{(p)}(\rho(f_n))}_{(p+1)}\norm{m}^p < \epsilon.
\end{equation*}
We then have for $n \geq N$,
\begin{align*}
\norm{\rho_{(1-m)}(f_n) - \rho(f_n)} &= \norm{\sum_{p \geq 1} \theta_p[m,\dotsc,m] \left(\delta_{A:B}^{(p)}(\rho(f_n))\right)}\\
&\leq \sum_{p = 1}^{(P-1)}\norm{m}^p \projnorm{\delta_{A:B}^{(p)}(\rho(f_n))}_{(p+1)} + \sum_{p \geq P} \norm{m}^p CR^{-p}\\
&< 2\epsilon.
\end{align*}
It follows that
\begin{equation*}
 \widetilde \rho_{(1-m)}(f) - \widetilde \rho(f) = \lim_{n \to \infty} \rho_{(1-m)}(f_n) - \rho(f_n) = 0.
\end{equation*}
Now let $S$ be a $(0,1)$-semicircular element in $M$ which is independent from $B$ and $B$-freely independent from $A$.  Take $\epsilon > 0$ sufficiently small so that $\norm{U_\epsilon - 1} < R$, where $U_\epsilon = \exp(i\pi \epsilon S)$.  Then $\widetilde \rho_{U_\epsilon}(f) = \widetilde \rho(f)$, in particular $\widetilde \rho(f) \in C^*(A \vee B) \cap C^*(U_\epsilon AU_\epsilon \vee B)$.  By Corollary \ref{split}, we have $\widetilde \rho(f) \in B$.
\end{proof}

\begin{rmk}\label{half}
We recall the following from \cite{coalgebra}.  If $A$ is a unital C$^*$-algebra, the \textit{upper half-plane} of $A$ is defined as $\mb H_+(A) = \{T \in A: \text{Im }T \geq \epsilon 1 \text{ for some }\epsilon > 0\}$.  Similarly, the \textit{lower half-plane} of $A$ is defined as $\mb H_-(A) = \{T \in A: \text{Im }A \leq -\epsilon 1 \text{ for some }\epsilon > 0\}$. If $T \in \mb H_+(A)$, then $T$ is invertible, and $T^{-1} \in \mb H_-(A)$.  Moreover,
\begin{align*}
 \norm{T^{-1}} &\leq \epsilon^{-1} & \text{Im}(T^{-1}) &\leq - \left(\epsilon + \epsilon^{-1}\norm{T}^2\right)^{-1}.
\end{align*}

\end{rmk}

\begin{prop}
Let $1 \in B \subset M$ be a W$^*$-subalgebra, and let $1 \in A, C \subset M$ be $*$-subalgebras.  Assume $A$ and $C$ are $B$-free in $(M,E_B)$.  Suppose also that $|j(A:B)|_2 < \infty$.  Then there is a holomorphic function $F:\mb H_+(A) \times \mb H_+(C) \to B$ such that
\begin{equation*}
 E_{A \vee B} (a +c)^{-1} = (a + F(a,c))^{-1}
\end{equation*}
for $a \in \mb H_+(A)$, $c \in \mb H_+(C)$.
\end{prop}

\begin{proof}
Let $a \in \mb H_+(A)$, $c \in \mb H_+(C)$, and let $\alpha = (a+c)^{-1}$.  By Proposition \ref{libmorph}, $|j(A:B \vee C)|_2 < \infty$, so $\delta_{A:B}$ and $\delta_{A:B \vee C}$ are closable in norm.  By Proposition \ref{closder}, $\alpha \in \frk D(\overline \delta_{A:B \vee C})$.  By Lemma \ref{derresolvents},
\begin{equation*}
\overline \delta_{A:B}(\alpha) = -\alpha(a\otimes 1 - 1\otimes a)\alpha.
\end{equation*}
It follows from Proposition \ref{deltamorph} that $\gamma = E_{B}(\alpha) \in \frk D(\overline \delta_{A:B})$ and
\begin{equation*}
 \overline \delta_{A:B}(\gamma) = -\gamma(a \otimes 1 - 1 \otimes a)\gamma.
\end{equation*}
Since $\alpha \in \mb H_-(M)$, it follows that also $\gamma \in \mb H_-(M)$, in particular $\gamma$ is invertible.  By Lemma \ref{derresolvents}, $\gamma = (a+n)^{-1}$ for some $n \in \Ker \overline \delta_{A:B}$. 

Setting $F(a,c) =n$, it is clear that $F(a,c)$ depends analytically on $(a,c)$, it remains only to show that $F(a,c) \in B$.  Fix $a \in \mb H_+(A)$ and denote $F_a(c) = F(a,c)$ for $c \in \mb H_+(C)$.  Since $F_a: \mb H_+(C) \to M$ is holomorphic, it suffices to show that $F_a(c) \in B$ for $c$ in some open subset of $\mb H_+(C)$.  

Fix $0 < R < 1$ and choose $x$ sufficiently large so that $2\norm{a}(1-R)^{-1}x^{-1} < 1/2$. Let 
\begin{equation*}
 \Omega = \{c \in \mb H_+(C): \norm{c - ix} < \norm{a}\}.
\end{equation*}
Given $c \in \Omega$, we have
\begin{equation*}
 (a+c)^{-1} = ((ix)(1-\Gamma))^{-1} = (ix)^{-1}\sum_{k \geq 0} \Gamma^k,
\end{equation*}
where $\Gamma = (ix)^{-1}(ix - a - c)$.  Note that $\norm{\Gamma} < 2\norm{a}x^{-1}$.  For $p \geq 1$ we have
\begin{equation*}
\delta_{A:B \vee C}^{(p)}(\Gamma) = (ix)^{-1}\left(1 \otimes a \otimes 1^{\otimes (p-1)} - a \otimes 1^{\otimes p}\right),
\end{equation*}
so that
\begin{equation*}
 \projnorm{\delta_{A:B\vee C}^{(p)}(\Gamma)}_{(p+1)} \leq 2 \norm{a}x^{-1}.
\end{equation*}
Letting $P \in A * B$ such that $\rho(P) = \Gamma$, it follows that $f \in A *_R (B \vee C)$ and
\begin{equation*}
\norm{P}_{A *_R (B \vee C)} < 2\norm{a}x^{-1}(1-R)^{-1} < 1/2.
\end{equation*}
Since $A *_R (B \vee C)$ is a Banach algebra, we have
\begin{equation*}
 \norm{P^k}_{A *_R (B \vee C)} < 2^{-k}
\end{equation*}
for $k \geq 1$.  Let $f_k \in A *B$ be such that $\rho(f_k) = E_{A \vee B}(\Gamma^k)$, by Proposition \ref{deltamorph} we have 
\begin{equation*}
 \norm{f_k}_R^\sim< 2^{-k}
\end{equation*}
for $k \geq 1$.  It follows that $\sum_{k \geq 1} f_k$ converges in $A \ttimes R B$ to a limit $f$ with $\norm{f}^\sim_R < 1$.  Let $g = (ix)^{-1}(1+f)^{-1} - a \in A \ttimes R B$, then
\begin{equation*}
F_a(c) = \widetilde \rho(f),
\end{equation*}
so that $\overline \delta_{A:B}(\widetilde \rho(f)) = \overline \delta_{A:B}(F_a(c)) = 0$.  By Proposition \ref{deltaker}, $F_a(c) \in B$.

\end{proof}

\begin{rmk}
We may now remove the condition on the liberation gradient.
\end{rmk}

\begin{thm}
Let $1 \in B \subset M$ be a $W^*$-subalgebra, and let $1 \in A, C \subset M$ be $*$-subalgebras.  Assume $A$ and $C$ are $B$-free in $(M,E_B)$.  Then there is a holomorphic function $F:\mb H_+(A) \times \mb H_+(C) \to B$ such that
\begin{equation*}
 E_{A \vee B} (a +c)^{-1} = (a + F(a,c))^{-1}
\end{equation*}
for $a \in \mb H_+(A)$, $c \in \mb H_+(C)$.
\end{thm}

\begin{proof}
Let $a \in \mathbb H_+(A), c \in \mathbb H_+(C)$, and set
\begin{equation*}
 F(a,c) = \left(E_{A \vee b}(a+c)^{-1}\right)^{-1} - a,
\end{equation*}
we must show that $F(a,c) \in B$.  Clearly $F(a,c)$ depends analytically on $(a,c)$, hence it suffices to show that $F(a,c) \in B$ for $(a,c)$ in some open subset of $\mb H_+(A) \times \mb H_+(C)$.  Let
\begin{equation*}
\Omega = \{(a,c) \in \mb H_+(A) \times \mb H_+(C): \norm{a - i} < 1/2, \norm{c - Ki} < 1/2\},
\end{equation*}
where $K \gg 0$.

Now let $S$ be a $(0,1)$-semicircular element in $M$ which is freely independent from $A \vee B \vee C$.  For $0 < \epsilon < 1$ let $U_\epsilon = \exp(i\pi\epsilon S)$.  By Proposition \ref{conjsemicirc}, $|\xi(U_\epsilon:\C)|_2 < \infty$.  Hence by Proposition \ref{libconj}, $|j(U_\epsilon A U_\epsilon^*:B)|_2 < \infty$.

So fix $(a,c) \in \Omega$, by the proposition there are $b_\epsilon \in B$ for $0 < \epsilon < 1$ such that
\begin{equation*}
 E_{U_\epsilon A U_\epsilon^* \vee B} (U_\epsilon a U_\epsilon^* + c)^{-1} = (a + b_\epsilon)^{-1}.
\end{equation*}
Now since $(a,c) \in \Omega$, we have
\begin{align*}
 \norm{U_\epsilon aU_\epsilon^* +c} &\leq K+2 & \text{Im}(U_\epsilon a U_\epsilon^* + c) \geq K.
\end{align*}
It follows from \ref{half} that
\begin{align*}
 \norm{(U_\epsilon a U_\epsilon^* + c)^{-1}} &\leq K^{-1} & \text{Im}(U_\epsilon a U_\epsilon^* + C)^{-1} \leq - (K + (K+2)^2/K)^{-1}.
\end{align*}
Therefore
\begin{align*}
\norm{(U_\epsilon a U_\epsilon^* + b_\epsilon)^{-1}} &\leq K^{-1} & \text{Im}(U_\epsilon a U_\epsilon^* + b_\epsilon)^{-1} &\leq -(K+(K+2)^2/K)^{-1}.
\end{align*}
Applying \ref{half} once more, we see
\begin{equation*}
 \text{Im}(U_\epsilon a U_\epsilon^* + b_\epsilon) \geq \left((K+(K+2)^2/K)^{-1} + (K+(K+2)^2/K)K^{-2}\right)^{-1} = \frac{K + \frac{(K+2)^2}{K}}{3 + \frac{K^2}{(K+2)^2} + \frac{(K+2)^2}{K^2}}.
\end{equation*}
For $K$ sufficiently large, this is greater than 2, from which it follows that $\text{Im}(b_\epsilon) > 1/2$ for $0 < \epsilon < 1$.  In this case, it follows from \ref{half} that
\begin{equation*}
 \norm{(a+b_\epsilon)^{-1}} \leq C,
\end{equation*}
for some finite constant $C$ which does not depend on $\epsilon$.  Hence
\begin{equation*}
 \lim_{\epsilon \to 0} \norm{(U_\epsilon a U_\epsilon^* + b_\epsilon)^{-1} - (a + b_\epsilon)^{-1}} = 0,
\end{equation*}
and therefore
\begin{equation*}
 \lim_{\epsilon \to 0} \norm{(a + b_\epsilon)^{-1} - E_{A \vee B} (U_\epsilon a U_\epsilon^* + b_\epsilon)^{-1}} = 0.
\end{equation*}

An application of \cite[Lemma 3.3]{mutual} shows that $A \vee B$, $C$ and $S$ are $B$-free, and another application shows that $A \vee B$, $U_\epsilon A U_\epsilon^* \vee B$, $U_\epsilon AU_\epsilon^* \vee B \vee C$ is a freely Markovian triple.  By \cite[Lemma 3.7]{mutual}, we have
\begin{equation*}
 E_{A \vee B}E_{U_\epsilon AU_\epsilon^* \vee B} E_{U_\epsilon A U_\epsilon^* \vee B \vee C} = E_{A \vee B} E_{U_\epsilon A U_\epsilon^* \vee B \vee C}.
\end{equation*}
We therefore have
\begin{align*}
 E_{A \vee B} (a+c)^{-1} &= \lim_{\epsilon \to 0} E_{A \vee B}(U_\epsilon aU_\epsilon^* + c)^{-1}\\
&= \lim_{\epsilon \to 0} E_{A \vee B} E_{U_\epsilon A U_\epsilon^* \vee B}(U_\epsilon a U_\epsilon^* + c)^{-1}\\
&= \lim_{\epsilon \to 0} E_{A \vee B}(U_\epsilon a U_\epsilon^* + b_\epsilon)^{-1}\\
&= \lim_{\epsilon \to 0} (a + b_\epsilon)^{-1}.
\end{align*}
It follows that $b_\epsilon$ converges as $\epsilon \to 0$ to 
\begin{equation*}
F(a,c) = \left(E_{A \vee B}(a+c)^{-1}\right)^{-1} - a,
\end{equation*}
hence $F(a,c) \in B$ which completes the proof.
\end{proof}

\section{Analytic subordination for $B$-free multiplicative convolution of unitaries}
\noindent In this section we use the derivation $d_{U:B}$ to prove the analytic subordination result for multiplication of $B$-freely independent unitaries, where $B$ is a general W$^*$-algebra of constants.

\begin{rmk}
Let $1 \in B \subset M$ be a W$^*$-subalgebra, and let $B\langle t \rangle$ denote the algebra of noncommutative polynomials with coefficients in $B$.  Given any $m \in M$, there is a unique homomorphism from $B\langle t \rangle$ into $M$ which is the identity on $B$ and sends $t$ to $m$, which we will denote by $f \mapsto f(m)$.
\end{rmk}

\begin{rmk}
Recall that if $U \in M$ is a unitary, $d_{U:B}:B\langle U,U^* \rangle \to B\langle U,U^*\rangle \otimes B\langle U,U^* \rangle$ is the derivation determined by
\begin{align*}
 d_{U:B}(U) &= 1 \otimes U,\\
d_{U:B}\left(U^*\right) &= -U^* \otimes 1,\\
d_{U:B}(b) &= 0, \ms \ms\ms\ms\ms\ms\ms\ms\ms\ms\ms\ms (b \in B).
\end{align*}
Define $d_{U:B}^{(p)}:B\langle U,U^*\rangle \to (B\langle U,U^* \rangle)^{\otimes (p+1)}$ recursively by $d_{U:B}^{(0)} = \text{id}$, 
\begin{equation*}
 d_{U:B}^{(p+1)} = (d_{U:B} \otimes \text{id}^{\otimes p})\circ d_{U:B}^{(p)}.
\end{equation*}
\end{rmk}

\begin{rmk}
Fix a unitary $U \in M$ which is algebraically free from $B$.  Define a norm $\vnorm^\sim_{R,U}$ on $B \langle t \rangle$ by
\begin{equation*}
 \norm{f}_{R,U}^\sim = \sum_{p \geq 0} \projnorm{d_{U:B}^{(p)}(f(U))}_{(p+1)},
\end{equation*}
where $\vprojnorm_{(s)}$ denotes the projective tensor product norm on $M^{\widehat \otimes s}$.
\begin{lem*}
$\vnorm_{R,U}^\sim$ is a finite norm on $B\langle t \rangle$, and if $f,g \in B \langle t \rangle$ then
\begin{equation*}
 \norm{fg}_{R,U}^\sim \leq \norm{f}_{R,U}^\sim \norm{g}_{R,U}^\sim.
\end{equation*}

\end{lem*}
\noindent The proof is the same as the argument for $\delta$ given in (\ref{deltanorm}).

\end{rmk}

\begin{rmk}
Let $B_{R,U}^\sim\{t\}$ denote the completion of $B\langle t \rangle$ under $\vnorm_{R,U}^\sim$.  The map sending $f \in B \langle t \rangle$ to $f(U)$ extends to a contractive homomorphism from $B_{R,U}^\sim$ into $M$, which we will still denote by $f \mapsto f(U)$.
\end{rmk}

\begin{rmk}
Similarly to $\delta$, $d_{U:B}$ is related to the homomorphism $f \mapsto f((1+m)U)$, $f \in B \langle t \rangle$, where $m \in M$ is fixed.  Recall that if $m_1,\dotsc,m_s \in M$ are given, $\theta_s[m_1,\dotsc,m_s]: M^{\otimes (s+1)} \to M$ is the linear map determined by
\begin{equation*}
 \theta_p[m_1,\dotsc,m_s](m'_1 \otimes \dotsb \otimes m'_{s+1}) = m'_1m_1 m'_2 \dotsb m_sm'_{s+1}.
\end{equation*}

\end{rmk}

\begin{prop}
Fix $m \in M$, then for $f \in B\langle t \rangle$ we have
\begin{equation*}
 f((1+m)U) = \sum_{p \geq 0} \theta_p[m,\dotsc,m]\left(d_{U:B}^{(p)}(f(U))\right).
\end{equation*}
In particular, if $\norm{m} \leq R$ then $f \mapsto f((1+m)U)$ extends to a contractive homomorphism from $B_{R,U}\{t\}$ into $M$, which we will also denote by $f \mapsto f((1+m)U)$.
\end{prop}

\begin{proof}
First observe that the right hand side has only finitely many nonzero terms, so convergence is not an issue.  Let $\varphi(f)$ denote the right hand side.  Repeating the argument from Proposition \ref{deltahom}, we see that $\varphi$ is a homomorphism from $B \langle t \rangle$ into $M$.  Since $\varphi(b) = b$ for $b \in B$, and 
\begin{equation*}
 \varphi(t) = (1+m)U,
\end{equation*}
it follows that $\varphi(f) = f((1+m)U)$ as claimed.  For $f \in B\langle t \rangle$ and $\norm{m} \leq R$, we then have
\begin{align*}
 \norm{f((1+m)U)} &\leq \sum_{p \geq 0} \norm{\theta_p[m,\dotsc,m]\left(d_{U:B}^{(p)}f(U)\right)}\\
&\leq \sum_{p \geq 0} \norm{m}^p\projnorm{d_{U:B}^{(p)}f(U)}_{(p+1)}\\
&\leq \norm{f}_{R,U}^\sim.
\end{align*}
So $f \mapsto f((1+m)U)$ extends to a contractive homomorphism on $B_{R,U}^\sim\{t\}$ as claimed.
\end{proof}

\begin{rmk}
Recall that $\xi(U:B)$ is determined by $\xi(U:B) \in L^1(W^*(B \langle U,U^* \rangle))$ and
\begin{equation*}
 \tau\left(\xi(U:B)m\right) = (\tau \otimes \tau)\left(d_{U:B}(m)\right) \ms \ms \ms \ms m \in B \langle U,U^* \rangle.
\end{equation*}
Voiculescu has proved that the existence of $\xi(U:B) \in L^2(B \langle U,U^* \rangle)$ is a sufficient condition for the closability of $d_{U:B}$ when viewed as an unbounded operator
\begin{equation*}
 d_{U:B}: L^2(W^*(B \langle U,U^* \rangle)) \to L^2(W^*(B \langle U,U^* \rangle) \otimes W^*(B \langle U,U^* \rangle)).
\end{equation*}
In particular, $|\xi(U:B)|_2 < \infty$ implies that $d_{U:B}$ is closable in the uniform norm, we will denote this closure by $\overline d_{U:B}$.
\end{rmk}

\begin{prop}\label{dker}
Suppose that $|\xi(U:B)|_2 < \infty$.  If $f \in B_{R,U}\{t\}$, then $f(U) \in \frk D(\overline d_{U:B})$.  Furthermore, if $R > 2$ and if $\overline d_{U:B}(f(U)) = 1 \otimes f(U)$, then $f(U) = Ub$ for some $b \in B$.
\end{prop}

\begin{proof}
Let $f \in B_{R,U}\{t\}$, it is clear from the definition of $\vnorm_{R,U}^\sim$ that $f(U) \in \frk D(\overline d_{U:B})$, suppose then that $\overline d_{U:B}(f(U)) = 1 \otimes f(U)$.  Let $f_n \in B\langle t \rangle$, $\norm{f_n -f }_{R,U}^\sim \to 0$.  Since $(d_{U:B} \otimes \text{id})$ is closable, we have
\begin{equation*}
\lim_{n \to \infty} d_{U:B}^{(2)}(f_n) = 0,
\end{equation*}
with convergence in $\vnorm_{(3)}$.  Iterating, we see that
\begin{equation*}
 \lim_{n \to \infty} d_{U:B}^{(p)}(f_n) = 0,
\end{equation*}
for all $p \geq 2$.  

Now let $m \in M$, $\norm{m} < R$.  Since $f_n \to f \in \vnorm_{R,U}^\sim$, there is a finite constant $C > 0$ such that $\norm{f_n}_{R,U}^\sim < C$ for all $n \in \N$.  Let $\epsilon > 0$ and find $P \geq 2$ such that 
\begin{equation*}
 C\frac{(\norm{m}/R)^P}{1-\norm{m}/R} < \epsilon.
\end{equation*}
Now find $N$ such that $n \geq N$ implies
\begin{equation*}
 \sum_{p = 2}^{P-1} \norm{m}^{p}\projnorm{d_{U:B}^{(p)}(f_n(U))}_{(p+1)} < \epsilon.
\end{equation*}
We then have, for $n \geq N$,
\begin{align*}
\norm{f_n((1+m)U) - \left(f_n(U) + \theta_1[m]\left(d_{U:B}(f_n(U))\right)\right)} < 2\epsilon.
\end{align*}
Taking limits, it follows that
\begin{equation*}
 f((1+m)U) = (1+m)f(U).
\end{equation*}
If $R > 2$, we can apply this to $m = U^* -1$ to find
\begin{equation*}
 f(1) = U^*f(U).
\end{equation*}
Since $f(1) \in B$, the result follows.
\end{proof}

\begin{rmk}\label{disk}
We will also use the following technical lemma from \cite{markovian}.

\begin{lem*}
 If $x \in A$, where $A$ is a unital C$^*$-algebra, the following are equivalent:
\begin{enumerate}
\renewcommand{\labelenumi}{(\roman{enumi})} 
\item $\norm{x} < 1$.
\item $1-x$ is invertible and $2 \text{Re}(1-x)^{-1} \geq (1 + \epsilon)$ for some $\epsilon > 0$.
\end{enumerate}

\end{lem*}

\end{rmk}

\begin{prop}
Let $1 \in B \subset M$ be a W$^*$-subalgebra, and let $U,V \in M$ be unitaries such that $B\langle U,U^*\rangle$ is $B$-freely independent from $B\langle V,V^* \rangle$ in $(M,E_B)$.  Suppose also that $|\xi(U:B)|_2 < \infty$.  Then there is a holomorphic map $F:\mb D(B) \to \mb D(B)$ such that
\begin{equation*}
E_{B \langle U,U^* \rangle} UVb(1-UVb)^{-1} = UF(b)(1-UF(b))^{-1}
\end{equation*}
and $\norm{F(b)} \leq b$ for $b \in \mb D(B)$.
\end{prop}

\begin{proof}
Since $|\xi(U:B)|_2 < \infty$, also $|\xi(UV:B)|_2 < \infty$ by \ref{conjmorph}.  So $d_{U:B}$ and $d_{UV:B}$ are both closable in the uniform norm.  Let $b \in \mb D(B)$, and set $\alpha = UVb(1-UVb)^{-1}$.  Then $\alpha \in \frk D(\overline d_{UV:B})$ by Proposition \ref{closder}, and by Lemma \ref{derresolvents} we have
\begin{equation*}
 \overline d_{UV:B} (\alpha) = (\alpha + 1) \otimes \alpha. 
\end{equation*}
It follows from Corollary \ref{dmorph} that $ \gamma = E_{B \langle U,U^* \rangle}(\alpha) \in \frk D(\overline d_{U:B})$, and
\begin{equation*}
 \overline d_{U:B}(\gamma) = (\gamma + 1)\otimes \gamma.
\end{equation*}
Now
\begin{equation*}
 \gamma + 1 = E_{B \langle U,U^* \rangle} (1-UVb)^{-1},
\end{equation*}
so to show that $\gamma + 1$ is invertible, it suffices to show that $0$ is not in the convex hull of the spectrum of $(1-UVb)^{-1}$.  Let $z \in \C$, then 
\begin{equation*}
 (1-UVb)^{-1} - z = (1-z + zUVb)(1-UVb)^{-1}
\end{equation*}
is invertible if $|z|\norm{b} < |1-z|$, in particular if $\text{Re}(z) < 1/2$.  So $\gamma + 1$ is invertible, and by Lemma \ref{derresolvents} we have $\gamma = Un(1-Un)^{-1}$ for some $n \in \text{Ker } \overline d_{U:B}$ such that $1-Un$ is invertible.  

It is clear that $n$ depends analytically on $b$, it remains to show that $n \in \mb D(B)$, and that $\norm{n} \leq \norm{b}$.  First we claim that $\norm{n} < 1$.  Since $U$ is unitary, it suffices to show that $\norm{Un} < 1$.  By Lemma \ref{disk}, it suffices to show that $1-Un$ is invertible, and $2 \text{Re}(1-Un)^{-1} \geq (1 + \epsilon)$ for some $\epsilon > 0$.  But we have
\begin{equation*}
 (1- Un)^{-1} = \gamma + 1 = E_{B \langle U,U^* \rangle} (1-UVb)^{-1},
\end{equation*}
and since $\norm{UVb} < 1$, applying Lemma \ref{disk} again shows that $2 \text{Re}(1-UVb)^{-1} \geq (1+ \epsilon)$ for some $\epsilon > 0$.  So $\norm{b} < 1$, and it then follows from analyticity that in fact $\norm{F(b)} \leq \norm{b}$.  Indeed, let $b \in \mb D(B)$, and let $\psi$ a bounded linear functional on $M$, then $z \mapsto \psi(F(z(b/\norm{b})))$ is an analytic function $\mb D(\C) \to \mb D(\C)$.  By Schwarz's lemma, $|\psi(F(z(b/\norm{b})))| \leq |z|$ for $z \in \mb D(\C)$.  Taking $z = \norm{b}$, we have $|\psi(F(b))| \leq \norm{b}$, since $\psi$ is arbitrary we have $\norm{F(b)} \leq \norm{b}$.

Finally we claim that $F(b) \in B$ for $b \in \mb D(B)$.  By analytic continuation, it suffices to show this for $\norm{b}$ sufficiently small.  Let $R > 2$, $0 < \epsilon < 1/2$ and let $b \in B$, $\norm{b}(1+R) < \epsilon$.  We have
\begin{equation*}
 UVb(1-UVb)^{-1} = \sum_{n \geq 1} (UVb)^n.
\end{equation*}
Now 
\begin{equation*}
 d_{UV:B}^{(p)} (UVb) = \begin{cases} UVb & p = 0\\ 1 \otimes UVb & p = 1\\0 & p \geq 2\end{cases}.
\end{equation*}
In particular, setting $f = tb \in B \langle t \rangle$ we have
\begin{equation*}
 \norm{f}^\sim_{R,UV} < \epsilon.
\end{equation*}
It follows that
\begin{equation*}
 \norm{f^n}^\sim_{R,UV} < \epsilon^n.
\end{equation*}
Now since $U$ and $V$ are $B$-free, it follows that
\begin{equation*}
 E_{B \langle U,U^* \rangle} (UVb)^n \in B \langle U \rangle,
\end{equation*}
so let $P_n \in B \langle t \rangle$ be such that
\begin{equation*}
 P_n(U) = E_{B \langle U,U^* \rangle}(UVb)^n.
\end{equation*}
By Corollary \ref{dmorph},
\begin{equation*}
 d_{U:B}^{(p)} P_n(U) = (E_{B \langle U,U^* \rangle})^{\otimes (p+1)} d_{UV:B}^{(p)}(UVb)^n.
\end{equation*}
In particular,
\begin{equation*}
 \norm{P_n}_{R,U}^\sim \leq \norm{f^n}_{R,UV}^\sim < \epsilon^n.
\end{equation*}
So $\sum_{n \geq 1} P_n$ converges in $B_{R,U}^\sim\{t\}$ to some limit $h$ with $\norm{h} < 1$.  It follows that $1 + h$ is invertible in $B_{R,U}^\sim\{t\}$, and 
\begin{equation*}
 UF(b) = g(U),
\end{equation*}
where $g = 1- (1+h)^{-1}$.  But $g \in B_{R,U}^\sim\{t\}$ and  $\overline d_{U:B}( g(U)) = 1 \otimes g(U)$, so by Proposition \ref{dker}, $g(U) = Ub$ for some $b \in B$.  Since $U$ is invertible, we have $F(B) = b \in B$, which completes the proof.

\end{proof}

\begin{rmk}
We may now remove the condition on the conjugate $\xi(U:B)$.
\end{rmk}

\begin{thm}
Let $1 \in B \subset M$ be a W$^*$-subalgebra, and let $U,V \in M$ be unitaries such that $B\langle U,U^* \rangle$ is $B$-freely independent from $B \langle V,V^* \rangle$ in $(M,E_B)$.  Then there is a holomorphic map $F:\mb D(B) \to \mb D(B)$ such that
\begin{equation*}
E_{B \langle U,U^* \rangle} UVb(1-UVb)^{-1} = UF(b)(1-UF(b))^{-1}
\end{equation*}
and $\norm{F(b)} \leq b$ for $b \in \mb D(B)$.
\end{thm}

\begin{proof}
Let $S$ be a $(0,1)$-semicircular element in $(M,\tau)$ which is freely independent with $B \langle U,V,U^*,V^* \rangle$.  Applying \cite[Lemma 3.3]{mutual} twice, we see that $B\langle U,U^* \rangle$, $B \langle U_\epsilon U, U^* U_\epsilon^* \rangle$, $B\langle U_\epsilon U V, V^*U^*U_{\epsilon^*}\rangle$ is a freely Markovian triple, where $U_\epsilon = \exp(\pi i\epsilon S)$.  By \cite[Lemma 3.4]{mutual}, we have
\begin{equation*}
 E_{B \langle U,U^* \rangle}E_{B \langle U_\epsilon U,U^*U_\epsilon^* \rangle} E_{B \langle U_\epsilon U V,V^* U^* U_{\epsilon}^* \rangle} = E_{B \langle U,U^* \rangle} E_{B \langle U_\epsilon UV,V^*U^*U_\epsilon^* \rangle}.
\end{equation*}
Now $B \langle U_\epsilon U,U^*U_\epsilon^*\rangle$ and $B \langle V,V^* \rangle$ are $B$-free, and $|\xi(U_\epsilon U:B)|_2 < \infty$ by Corollary \ref{conjsemicirc}.  So given $b \in \mb D$, we may apply the proposition to find $n_\epsilon \in B$, $0 < \epsilon < 1$, such that $\norm{n_\epsilon} \leq \norm{b}$ and
\begin{equation*}
 E_{B \langle U_\epsilon U,U^*U_\epsilon^* \rangle} U_\epsilon UVb(1-U_\epsilon UVb)^{-1} = U_\epsilon Un_\epsilon(1-U_\epsilon Un_\epsilon)^{-1}.
\end{equation*}
It follows that
\begin{equation*}
E_{B \langle U,U^* \rangle} U_\epsilon UVb(1-U_\epsilon UVb)^{-1} = E_{B \langle U,U^* \rangle}U_\epsilon Un_\epsilon(1-U_\epsilon U n_\epsilon)^{-1}.
\end{equation*}
Now since $U_\epsilon UV b$ tends to $UVb$ as $\epsilon \to 0$, and $(1-UVb)^{-1}$ is invertible, it follows that
\begin{align*}
 \lim_{\epsilon \to 0} U_\epsilon UVb(1-U_\epsilon UVb)^{-1} &= \lim_{\epsilon \to 0} (1-U_\epsilon UVb)^{-1} - 1\\
&= (1- UVb)^{-1} - 1\\
&= UVb(1-UVb)^{-1},
\end{align*}
with convergence in norm.  Since $\norm{n_\epsilon} \leq \norm{b} < 1$ for $0 < \epsilon < 1$, it follows that
\begin{equation*}
 \lim_{\epsilon \to 0} \norm{U_{\epsilon}Un_\epsilon(1-U_\epsilon Un_\epsilon)^{-1} - Un_\epsilon(1-Un_\epsilon)^{-1}} = 0.
\end{equation*}
Hence,
\begin{align*}
 E_{B\langle U,U^* \rangle} UVb(1-UVb)^{-1} &= \lim_{n \to \infty} E_{B \langle U,U^* \rangle} U_\epsilon UVb(1-U_\epsilon UVb)^{-1}\\
&= \lim_{n \to \infty} E_{B \langle U,U^* \rangle} U_\epsilon U n_\epsilon(1-U_\epsilon Un_\epsilon)^{-1}\\
&= \lim_{n \to \infty} Un_\epsilon(1-Un_\epsilon)^{-1}.
\end{align*}
By the argument in the previous proposition, $E_{B \langle U,U^* \rangle} (1-UVb)^{-1}$ is invertible, so that
\begin{equation*}
 \lim_{\epsilon \to 0} 1 - Un_\epsilon = \left(E_{B\langle U,U^* \rangle} (1-UVb)^{-1}\right)^{-1}.
\end{equation*}
From this it follows that $n_\epsilon$ converges to a limit $n \in B$, such that $\norm{n} \leq \norm{b}$ and
\begin{equation*}
 E_{B \langle U,U^* \rangle}UVb(1-UVb)^{-1} = Un(1-Un)^{-1}.
\end{equation*}
Since the analytic dependence is clear, this completes the proof.

\end{proof}

\bibliographystyle{hsiam}
\bibliography{ref}

\end{document}